\documentclass{rmmcart}

\usepackage{color,amsmath,amssymb,amsthm,fullpage}

\newtheorem{theorem}{Theorem}
\newtheorem{lemma}[theorem]{Lemma}
\newtheorem{cor}[theorem]{Corollary}
\newtheorem{conj}[theorem]{Conjecture}

\theoremstyle{definition}
\newtheorem{definition}{Definition}

\newcommand{\Z}{\mathbb{Z}}

\title{Sequences of Consecutive Happy Numbers in\\ Negative Bases}
 
 \author{Helen G. Grundman}
  \address{Department of Mathematics, Bryn Mawr College}
 \email{grundman@brynmawr.edu}
 \author{Pamela E. Harris}
 \address{Department of Mathematics and Statistics, Williams College} 
 \email{pamela.e.harris@williams.edu}
 
\thanks{This work was supported in part by the National Science Foundation grant DMS-1545136.}

\begin{document}

\begin{abstract}
For $b\leq -2$ and $e \geq 2$, let $S_{e,b}:\Z\to\Z_{\geq 0}$ be the function taking an integer to the sum of the $e$-powers of the digits of its base $b$ expansion.  An integer $a$ is a {\em $b$-happy number} if there exists $k\in\mathbb{Z}^+$ such that $S_{2,b}^k(a) = 1$.
We prove that an integer is $-2$-happy if and only if it is congruent to 1 modulo 3 and that it is $-3$-happy if and only if it is odd.  Defining a $d$-sequence to be an arithmetic sequence with constant difference $d$ and setting $d = \gcd(2,b - 1)$, we prove that if $b \leq -3$ odd or $b \in \{-4,-6,-8,-10\}$, there exist arbitrarily long finite sequences of $d$-consecutive $b$-happy numbers.
\end{abstract}

\maketitle

\section{Introduction}
As is standard, a positive integer $a$ can be uniquely expanded in the base $b\geq 2$ as $a=\sum_{i=0}^na_ib^i$, where $0\leq a_i\leq b-1$ and $a_n\neq 0$. This definition can be extended to negative bases $b\leq -2$ in an analogous manner, with $0 \leq a_i \leq |b| - 1$.  The study of negative bases was introduced in the 1885 work of Vittorio Gr\"{u}nwald \cite{VG}.  As with positive bases, positive integers have unique representation in negative bases, but the same holds for negative numbers in negative bases, with no need for a negative sign: Any number written in a negative base with an odd number of digits is necessarily positive, while any written with an even number of digits is necessarily negative.
For example, converting between base 10 and base $-10$, we have $2017=(18197)_{(-10)}$ and $-2017=(2023)_{(-10)}$.

We begin by adapting the definition of generalized happy numbers and the corresponding function given in~\cite{GT01} to include the case of negative bases.  It is natural, in this case, to extend the domain of the function to include all integers.  

\begin{definition}
Let $b \leq -2$ and $e\geq 2$ be integers, and let $a\in \Z-\{0\}$ be given by $a=\sum_{i=0}^na_ib^i$ where $0\leq a_i\leq |b|-1$, for each $0\leq i\leq n$. Define the function $S_{e,b}:\Z\to\Z_{\geq 0}$ by $S_{e,b}(0) = 0$ and, for $a \neq 0$, 
\[S_{e,b}(a)=\sum_{i=0}^{n}a_i^e.\]
\end{definition}

Note that if $b \leq -2$ and $k > 0$, for each $a\neq 0$, $S_{e,b}^k(a) > 0$.  
\begin{definition}
An integer $a$ is an \emph{$e$-power $b$-happy number} if, for some $k \in \Z^+$, $S_{e,b}^k(a) = 1$.  A \emph{$b$-happy number} is a $2$-power $b$-happy number.
\end{definition}

The following definition and the others in the sequel are either directly from or adapted from~\cite{GT07}.
\begin{definition}
A {\em $d$-consecutive} sequence is defined to be an arithmetic sequence with
constant difference $d$.
\end{definition}

In Section~\ref{cyclesection}, we determine the fixed points and cycles of the functions $S_{2,b}$ for $-10 \leq b \leq -2$.  In Section~\ref{sequencesection}, we generalize the work of El-Sedy and Siksek~\cite{siksek}
and the work of Grundman and Teeple~\cite{GT07} 
on sequences of consecutive $b$-happy numbers. In particular, Grundman and Teeple  showed
that there exist arbitrarily long finite $d$-consecutive sequences of $b$-happy numbers, where $b\geq 2$ and $d=\gcd(2,b-1)$~\cite[Corollary 2]{GT07}. We prove that this result does not hold for $b = -2$, but does hold for all odd negative bases and for even negative bases $-10 \leq b \leq -4$.  

\section{Cycles and Fixed Points}
\label{cyclesection}

In this section, we first determine a bound,
dependent on the given base $b \leq -2$, 
such that each fixed
point and at least one point in every cycle is less than
this bound.  We then use this result to compute
all cycles and fixed points of $S_{2,b}$ for 
$-10 \leq b \leq -2$.

\begin{theorem}\label{thm:GetSmaller}
Let $b\leq -2$. If $a > (|b|-1)(|b|^2 - |b| + 1)$, 
then $0 < S_{2,b}(a) < a$.
\end{theorem}

\begin{proof}
Let $a$ and $b$ be as in the hypothesis.
Then $a = \sum_{i=0}^na_ib^i$ with $n$ even, $0\leq a_i\leq |b| - 1$, $a_n\neq 0$.
Observe that 
\begin{align}
a - S_{2,b}(a) &=\sum_{i=0}^na_ib^i-\sum_{i=0}^na_i^2
=\sum_{i=0}^{n}a_i(b^i-a_i)\nonumber \\
&=\sum_{j=1}^{\frac{n}{2}}a_{2j}(|b|^{2j}-a_{2j})-\sum_{j=1}^{\frac{n}{2}}a_{2j-1}\left(|b|^{2j-1} + a_{2j-1}\right)
+ a_0(|b|^0 - a_0).
\label{boundeq1}
\end{align}

{\bf Case: $\mathbf{n \geq 4}$.}  
Since, $a_n \geq 1$ and, for each $i$, $0\leq a_i\leq |b| - 1$, minimizing each term in~(\ref{boundeq1}) yields
\[a - S_{2,b}(a) \geq 1(b^n - 1) - \sum_{j=1}^{\frac{n}{2}}(|b| - 1)\left(|b|^{2j-1} + (|b| - 1)\right) + (|b| - 1)(1 - (|b| - 1)).\]
Noting that 
\[\sum_{j=1}^{\frac{n}{2}} |b|^{2j-1} = \frac{|b|}{|b|^2 - 1} \left(|b|^n - 1\right), \]
we have
\begin{align}
a - S_{2,b}(a) &\geq (b^n - 1) - (|b| - 1)\left(\frac{|b|}{|b|^2 - 1} \left(|b|^n - 1\right)+ \frac{n}{2}(|b| - 1)\right) - (|b| - 1)(|b| - 2)  \nonumber \\
&= \frac{b^n - 1}{|b| + 1} - \frac{n}{2}(|b| - 1)^2 - (|b| - 1)(|b| - 2)  \nonumber \\
&= \frac{1}{|b| + 1} \left(|b|^n - \frac{n}{2}(|b|^2 - 1)(|b| - 1) - (|b|^2 - 1)(|b| - 2) - 1\right)
\label{b=2} \\
&> \frac{1}{|b| + 1} \left(|b|^n - \frac{n}{2}|b|^3 - |b|^3  \right) \label{b>2} \\
&>  \frac{1}{|b| + 1} \left(|b|^{n-3} - \frac{n}{2} - 1 \right). \label{finalineq}
\end{align}

Note that the function $f(x) = 2^{x-3} - x/2 - 1$ is an increasing function for $x \geq 5$ and that $f(5) > 0$.  Thus, for $n \geq 5$, since $b \leq -2$, 
\[|b|^{n-3} - n/2 - 1 \geq 2^{n-3} - n/2 - 1 > 0,\]
and so, by inequality~(\ref{finalineq}), $a - S_{2,b}(a) > 0$.
Now, for $n = 4$ and $b < -2$, using inequality~(\ref{b>2}),
$a - S_{2,b}(a) > \frac{1}{|b| + 1} \left(|b|^4 - 3|b|^3  \right) \geq 0$, and for
$n = 4$ and $b = -2$, using inequality~(\ref{b=2}), $a - S_{2,b}(a) > 0$.

{\bf Case: $\mathbf{n < 4}$.}  
In this case, $(|b| - 1)(|b|^2 - |b| + 1) < a \leq (|b|-1)(|b|^2+1)$.  So $a = a_2b^2 + a_1 b + a_0$ with $a_2 = |b| - 1$, $0 \leq a_1 \leq |b| - 2$, and $0 \leq a_0 \leq |b| - 1$.
Thus, 
\begin{align*}
a - S_{2,b}(a) &=
a_{2}(|b|^{2}-a_{2}) - a_{1}\left(|b| + a_{1}\right)
+ a_0(1 - a_0) \\
&\geq (|b| - 1)(|b|^2 - (|b| - 1)) - (|b| - 2)(|b| + (|b| - 2)) + (|b| - 1)(1 - (|b| - 1)) \\
& = |b|^3 - 5|b|^2 + 11|b| - 7 > 0,
\end{align*}
since $b \leq -2$.
\end{proof}

Note that when $b = -2$, the bound in Theorem~\ref{thm:GetSmaller} is 4.  Since 3 is a fixed point of $S_{2,-2}$, the given bound is best possible. 

The following corollary is immediate.

\begin{cor} \label{bound}
Let $b \leq -2$.  Every fixed point of $S_{2,b}$ is
less than or equal to $(|b| - 1)(|b|^2 - |b| + 1)$ and every cycle of $S_{2,b}$ contains a number that is less than or equal to $(|b| - 1)(|b|^2 - |b| + 1)$.
\end{cor}

Using Corollary~\ref{bound} and a direct computer search, we determine all fixed points and cycles in the bases $-10 \leq b \leq -2$.  The results are given in Table \ref{tab:1}.

\begin{table}[ht]
\centering
\begin{tabular}{|c|p{1cm}|p{10.5cm}|p{1.5cm}|p{1.5cm}|}\hline
Base& Fixed Points & Cycles& Smallest happy number $> 1$& Largest happy number $< -1$\\\hline\hline
$-2$&1,2,3& None&4&-2\\\hline
$-3$&1&(2,4,6)&3&-1\\\hline
$-4$&1&(6,14)&16&-4\\\hline
$-5$&1,10,11&(2,4,16,6,18,14,26), (9,33,29,17)&25&-5\\\hline
$-6$&1&(2,4,16,33,11,51,29,30)&36&-6\\\hline
$-7$&1, 41&(2,4,16,30,14,26,42), (5,25,33,35), (6,36)&49&-7\\\hline
$-8$&1, 46&(11,59), (30,62,38,53)&64&-8\\\hline
$-9$&1&(6,36,26,114,76,18,50,42,62,74), (9,65), (27,37)&5&-5\\\hline
$-10$&1&(19,163,29,146,76,46,73), (35,75)&100&-10\\\hline
\end{tabular}
\vspace{1pt}
\caption{Base 10 representation of fixed points and cycles of $S_{2,b}$ for $-10 \leq b \leq -2$.}
\label{tab:1}
\end{table}
\begin{definition}
\label{Udef}
For $e \geq 2$ and $b \leq -2$, let
\[U_{e,b} = \{a \in \Z^+ | \mbox{ for some } m\in \Z^+,\ S_{e,b}^m(a) = a\}.\]
\end{definition}

The following straightforward lemmas are used throughout this work.   
\begin{lemma}\label{Ulemma}
Fix $b \leq -2$.  For each $a\neq 0$, there exists some $k \in \Z^+$ such that 
$S_{2,b}^k(a) \in U_{2,b}$.
\end{lemma}

\begin{lemma}\label{paritylemma}
Fix $b \leq -2$, $a\in \Z$, and $k\in \Z^+$. 
If $b$ is odd, then 
\[S_{2,b}^k(a) \equiv a \pmod 2.\]
\end{lemma}

\begin{proof}
Fix $a$, $b$, and $k$ as in the lemma.  Noting that the result is trivial if $a = 0$, let $a = \sum_{i=0}^{n}a_ib^i$.  Then
\[a = \sum_{i=0}^{n}a_ib^i \equiv \sum_{i=0}^{n}a_i \equiv \sum_{i=0}^{n}a_i^2 \equiv S_{2,b}(a) \pmod 2. \]  A simple induction argument completes the proof.
\end{proof}

\section{Consecutive $\mathbf b$-Happy Numbers}
\label{sequencesection}

In this section, we consider sequences of 
consecutive $b$-happy numbers for negative $b$.  
Grundman and Teeple~\cite{GT07} proved, for each base $b \geq 2$, that, letting $d = \gcd(2,b-1)$, there
exist arbitrarily long finite sequences of $d$-consecutive $b$-happy numbers.    We prove the following theorem using ideas from both of~\cite{siksek,GT07}.  Note that part (1) of the theorem demonstrates that the results in~\cite{GT07} do not generalize directly to negative bases.

\begin{theorem}
\label{seqthm}
Let $b \leq -2$.
\begin{enumerate}
\item
There is an infinitely long sequence of $3$-consecutive $-2$-happy numbers.  In particular, $a\in \Z^+$ is $-2$-happy if and only if $a \equiv 1 \pmod 3$.
\item
There is an infinitely long sequence of $2$-consecutive $-3$-happy numbers.  In particular, $a\in \Z^+$ is $-3$-happy if and only if $a \equiv 1 \pmod 2$.
\item
For $b\in \{-4,-6,-8,-10\}$, there exist arbitrarily long finite sequences of consecutive $b$-happy numbers.
\item
For $b$ odd, there exist arbitrarily long finite sequences of $2$-consecutive $b$-happy numbers.
\end{enumerate}
\end{theorem}

The smaller even negative bases not covered by Theorem~\ref{seqthm} are addressed in the following conjecture.

\begin{conj}
For $b\leq -12$ and even, there exist arbitrarily long finite sequences of consecutive $b$-happy numbers.
\end{conj}

We begin by proving the first two cases of Theorem~\ref{seqthm}.  The other two cases follow immediately from Corollary~\ref{finalcor}, stated and proved at the end of this section.

\begin{lemma}
\label{-2-3}
A positive integer $a$ is $-2$-happy if and only if $a \equiv 1 \pmod 3$ and is $-3$-happy if and only if $a$ is odd.
\end{lemma}

\begin{proof}If $a=\sum_{i=0}^{n}a_i(-2)^i$ with $a_i\in\{0,1\}$ for all $0\leq i \leq n$, then
\[S_{2,-2}(a) = \sum_{i=0}^{n}a_i^2 = \sum_{i=0}^{n}a_i \equiv 
\sum_{i=0}^{n}a_i(-2)^i \equiv a \pmod 3.\]
Thus, if $a$ is $-2$-happy, $a \equiv 1 \pmod 3$.
Now, suppose that $a \equiv 1 \pmod 3$.  By Lemma~\ref{Ulemma}, there exists a $k\in \Z^+$ such that  $S_{2,-2}^k(a)\in U_{2,-2} = \{1,2,3\}$. Since $S_{2,-2}^k(a)\equiv a \equiv 1 \pmod 3$, $S_{2,-2}^k(a) = 1$ and so $a$ is a $-2$-happy number.  

By Lemma~\ref{paritylemma}, if $a$ is a $-3$-happy number, then $a$ is odd.  Since $U_{2,-3} = \{1,2,4,6\}$,  Lemmas~\ref{Ulemma} and~\ref{paritylemma} together imply that if $a \equiv 1 \pmod 2$, then $a$ is a $-3$-happy number.
\end{proof}

The following definitions are from~\cite{GT07}.

\begin{definition}
\label{good}
Let $e \geq 2$ and $b \leq -2$.
A finite set $T$ is {\em $(e,b)$-good}
if, for each $u \in U_{e,b}$, there exist
$n$, $k \in {\Z}^+$ such that for each $t \in T$, $S_{e,b}^k(t+n) = u$.
\end{definition}

\begin{definition}
Let $I: {\Z}^+ \rightarrow {\Z}^+$ be defined by $I(t) = t + 1$.
\end{definition}

We will prove that for each odd $b \leq -5$ and for $b\in \{-4,-6,-8,-10\}$,
a finite set $T$ of positive integers is $(2,b)$-good if and only if
all of the elements of $T$ are congruent modulo $d = \gcd(2,b-1)$. 
Lemma~\ref{Flemma} and its proof are analogous to~\cite[Lemma 4 and proof]{GT07}.

\begin{lemma}
\label{Flemma}
Fix $e \geq 2$ and $b \leq -2$.  Let $T \subseteq \Z^+$ be finite.  Let $F:{\Z}^+ \rightarrow {\Z}^+$ be the
composition of a finite sequence of the functions $S_{e,b}$ and $I$.
If $F(T)$ is $(e,b)$-good, then $T$ is $(e,b)$-good.
\end{lemma}
\begin{proof}
Fix $e \geq 2$, $b \leq -2$, and a finite set of positive integers, $T$.  Clearly,
if $I(T)$ is $(e,b)$-good, then $T$ is $(e,b)$-good.  Using a simple
induction argument, it suffices to show that if
$S_{e,b}(T)$ is $(e,b)$-good, then $T$ is $(e,b)$-good.

Let $S_{e,b}(T)$ be
$(e,b)$-good and $u \in U_{e,b}$.  Then, by the definition of $(e,b)$-good, there exist $n^\prime$
and $k^\prime$ such that for each $s \in S_{e,b}(T)$,
$S_{e,b}^{k^\prime}(s + n^\prime) = u$.  
Let $\ell$ be the number of base $b$ digits of the largest element of $T$ and let
$\ell^\prime = \ell$ or $\ell + 1$ such that $n^\prime + \ell^\prime$ is odd.
Let \[n = \sum_{i = \ell^\prime}^{n^\prime + \ell^\prime -1} b^i =
\underbrace{11\ldots11}_{n^\prime}\underbrace{00\ldots00}_{\ell^\prime} \in \Z^+.\]
 Then $S_{e,b}(n) = n^\prime$ and
for each $t \in T$, $S_{e,b}(t+n) = S_{e,b}(t) + n^\prime$.
Let $k = k^\prime +1$.
Then for each $t\in T$,
\[S_{e,b}^{k}(t+n) = S_{e,b}^{k^\prime}(S_{e,b}(t+n)) =
S_{e,b}^{k^\prime}(S_{e,b}(t) + n^\prime) = u.\]
So $T$ is $(e,b)$-good. 
\end{proof}

\begin{lemma}\label{Flemmas}
Let $b\in \{-4,-6,-8,-10\}$   
and let $0 < t_2 < t_1$ be integers. Then there exists 
a function $F$ of the type described in
Lemma~\ref{Flemma} such that $F(t_1) = F(t_2)$.
\end{lemma}

\begin{proof}
Let $k\in \Z^+$ such that $S_{2,b}^{k}(t_1), S_{2,b}^{k}(t_2) \in U_{2,b}$, and let $F_1 = S_{2,b}^k$.
If $F_1(t_1) = F_1(t_2)$, we are done, and so we assume otherwise. From Table~\ref{tab:1}, we have
\begin{align*}
U_{2,-4} &= \{1,6,14\},\\
U_{2,-6} &= \{1,2,4,11,16,29,30,33,51\},\\
U_{2,-8} &= \{1,11,30,38,46,53,59,62\},\\
U_{2,-10} &=\{1,19,29,35,46,73,75,76,146,163\}.
\end{align*}

{\bf Case: $\mathbf{b = -4}$.}  
Let $F_2 = S_{2,-4}^2I$ and $F_3 = S_{2,-4}^5I^3$.  Note that
\begin{align*}
F_2(6) &= S_{2,-4}^2(7) = S_{2,-4}^2(133_{(-4)})
= S_{2,-4}(19) = S_{2,-4}(103_{(-4)}) = 10
\text{ and} \\
F_2(1) &= S_{2,-4}^2(2) = S_{2,-4}(4) = S_{2,-4}(130_{(-4)}) = 10.
\end{align*}
Thus, if $\{F_1(t_1),F_1(t_2)\} = \{1,6\}$, then let $F = F_2F_1$, so that $F(t_1) = F(t_2)$.  And if $\{F_1(t_1),F_1(t_2)\} = \{1,14\}$, then, noting that $S_{2,-4}(14) = 6$, let $F = F_2F_1S_{2,-4}$. 

Finally, observe that
\begin{align*}
F_3(6) &= S_{2,-4}^5(9) = S_{2,-4}^5(121_{(-4)})
= S_{2,-4}^4(6) = S_{2,-4}^4(132_{(-4)}) = S_{2,-4}^3(14) = S_{2,-4}^2(6) = 6
\text{ and} \\
F_3(14) &= S_{2,-4}^5(17) = S_{2,-4}^5(101_{(-4)}) = S_{2,-4}^4(2) = S_{2,-4}^3(4)  = S_{2,-4}^2(10) = S_{2,-4}(9) = 6.
\end{align*}
Hence, if $\{F_1(t_1),F_1(t_2)\} = \{6,14\}$, let $F = F_3F_1$.

{\bf Case: $\mathbf{b = -6}$.}  
Let $F_2 = S^\ell_{2,-6} I^{36-F_1(t_1)}$ where $\ell\in \Z^+$ such that  $F_2F_1(t_2) \in 
U_{2,-6}$. Note that $F_2F_1(t_1) = S^\ell_{2,-6} (36) = 1$, regardless of the choice of $\ell$.  If $F_2F_1(t_2) = 1$, we are done.  If not, since $(2,4,16,33,11,51,29,30)$ is a cycle, we can modify our choice of $\ell$ (making it larger, if necessary) to guarantee that $F_2F_1(t_2) = 2$.

Now let $F_3 = S_{2,-6}^6I^{7}$.  Noting that $F_3(1) = F_3(2)$, we set $F = F_3F_2F_1$. 

{\bf Case: $\mathbf{b = -8}$.}  
Let $F_2 = S^\ell_{2,-8} I^{64-F_1(t_1)}$ where $\ell\in \Z^+$ such that  $F_2F_1(t_2) \in 
U_{2,-8}$. Since $F_2F_1(t_1) = S^\ell_{2,-8} (64) = 1$, if   $F_2F_1(t_2) = 1$, we are done.  Otherwise, using Table~\ref{tab:1}, we can choose a possibly larger value of $\ell$ so that $F_2F_1(t_2) \in\{30,59,46\}$. If $F_2F_1(t_2) \in \{30,59\}$, let $F_3 = S_{2,-8}^8I^{2}$.  Noting that $F_3(1) = F_3(30) = F_3(59)$, we set $F = F_3F_2F_1$. If instead $F_2F_1(t_2)= 46$, then let $F_4 = S_{2,-8}^9I^7$. Since $F_4(1) = F_4(46)$, setting $F = F_4F_2F_1$ completes this case.  

{\bf Case: $\mathbf{b = -10}$.}  
Let $F_2 = S^\ell_{2,-10} I^{100-F_1(t_1)}$ where $\ell\in \Z^+$ such that  $F_2F_1(t_2) \in 
U_{2,-10}$. Since $F_2F_1(t_1) = 1$, if $F_2F_1(t_2) =S^\ell_{2,-10}(100)= 1$, we are done.  If not, we can choose $\ell$ so that $F_2F_1(t_2) \in \{19,35\}$. If $F_2F_1(t_2) = 19$, let $F_3 = S_{2,-10}^3I^{22}$ and set $F = F_3F_2F_1$.
If instead $F_2F_1(t_2) =35$, let $F_4 = S_{2,-10}^{16}I$ 
and set $F = F_4F_2F_1$, completing the proof.  
\end{proof}

We now apply the methods in~\cite{GT07} to odd negative~bases, noting that the original proof does not carry over, since, for $b$ negative, $b - 1 \neq |b| - 1$.

\begin{lemma}
\label{oddcongruencelemma}
Fix $b \leq -5$ odd, $v^\prime \in 2\Z^+$, and $r^\prime \in \Z^+$ such
that $b^{2r^\prime} > v^\prime$.
There exists $0 \leq c < |b| - 1$ such that
\begin{equation}
\label{congruence}
2c \equiv 4r^\prime - S_{2,b}\left((|b| - 1)\sum_{i=0}^{r^\prime - 1} b^{2i+1} + v^\prime - 1\right) - 1 \pmod{b - 1}.
\end{equation}
\end{lemma}

\begin{proof}
Since $b$ is odd and $v^\prime$ is even, the input to $S_{2,b}$ 
in~(\ref{congruence}) is odd.  Thus, by Lemma~\ref{paritylemma},  the output is also odd.
Hence, we can choose 
\[c \equiv 2r^\prime - \frac{1}{2}\left( S_{2,b}\left((|b| - 1)\sum_{i=0}^{r^\prime - 1} b^{2i+1} + v^\prime - 1\right) + 1\right) \; \left(\bmod\; {\frac{b - 1}{2}}\right),\]
with $0 \leq c < |\frac{b - 1}{2}| < |b| - 1$, since $b \leq -5$.
\end{proof}

\begin{lemma}
\label{Ftheoremodd}
Fix $b \leq -5$ odd and let $t_1$, $t_2\in \Z^+$ be congruent 
modulo $2$ with $t_2 < t_1$. Then there exists 
a function $F$ of the type described in
Lemma~\ref{Flemma} such that $F(t_1) = F(t_2)$.
\end{lemma}
\begin{proof}
First note that if
$t_1$ and $t_2$ have the same non-zero digits, then
$S_{2,b}(t_1) = S_{2,b}(t_2)$, and so  
$F = S_{2,b}$ suffices.

Next, if  $t_1 \equiv t_2 \pmod{b - 1}$, let $v\in \Z^+$ such 
that $t_2 - t_1 = (b - 1)v$.
Choose $r \in \Z^+$ so that $b^{2r} > b^2v + t_1$, and
let $m = b^{2r} + v - t_1 > 0$.
Then \[I^m(t_1) = t_1 + b^{2r} + v - t_1 = 
b^{2r} + v\]
and
\[I^m(t_2)= t_2 + b^{2r} + v - t_1 = b^{2r} + v+(b-1)v=b^{2r}+bv.\]
Since $b^{2r} > b^2v$, it follows that $I^m(t_1)$ and $I^m(t_2)$
have the same non-zero digits. Thus, it suffices
to let $F = S_{2,b} I^{m}$.

Finally, if $t_1 \not\equiv t_2 \pmod{b - 1}$, let
$v^\prime = t_1 - t_2 \in 2\Z^+$.
Choose $r^\prime \in \Z^+$ such
that $b^{2r^\prime} > b^2t_1$.
By Lemma~\ref{oddcongruencelemma}, since $b^2t_1 > v^\prime$, there exists $0 \leq c < |b| - 1$ such that
congruence~(\ref{congruence}) holds.

Let \[{m^\prime} = 
cb^{2r^\prime } + 
\sum_{i=0}^{r^\prime - 1} (|b| - 1)b^{2i} - t_2 \geq 0.\]
Then 
\[S_{2,b}(t_2 + m^\prime) = 
S_{2,b}\left(cb^{r^\prime } +
\sum_{i=0}^{r^\prime - 1} (|b| - 1)b^{2i}\right) = 
c^2 + r^\prime(|b| - 1)^2.\]
And
\begin{align*}
S_{2,b}\left(t_1 + m^\prime\right) &= 
S_{2,b}\left(cb^{2r^\prime } +
\sum_{i=0}^{r^\prime - 1} (|b| - 1)b^{2i} + v^\prime\right)\\
&= S_{2,b}\left((c + 1)b^{2r^\prime } +
\sum_{i=0}^{r^\prime - 1} (|b| - 1)b^{2i+1} + v^\prime - 1\right)\\
&= (c + 1)^2 + S_{2,b}\left(\sum_{i=0}^{r^\prime - 1} (|b| - 1)b^{2i+1} + v^\prime - 1\right).
\end{align*}

It follows that
\[S_{2,b}(t_1 + m^\prime) -
S_{2,b}(t_2 + m^\prime) = 
2c + 1 + S_{2,b}\left(\sum_{i=0}^{r^\prime - 1} (|b| - 1)b^{2i+1} + v^\prime - 1\right) - r^\prime(b + 1)^2.\]
Using congruence~(\ref{congruence}), 
this yields
\[S_{2,b}(t_1 + m^\prime) -
S_{2,b}(t_2 + m^\prime) \equiv 4r^\prime  - r^\prime(b + 1)^2
\equiv 0 \pmod{b - 1}.\]

Therefore,
$S_{2,b} (I^{m^\prime}(t_1)) \equiv S_{2,b}(I^{m^\prime}(t_2))\ \pmod{b-1}$.
Applying the earlier argument to these two numbers,
yielding an appropriate value of $m \in \Z^+$, we let $F = S_{2,b} I^{m} S_{2,b} I^{m^\prime}$.
\end{proof}

\begin{theorem} 
\label{mainthm}
Fix $b \leq -5$ odd or $b \in \{-4,-6,-8,-10\}$.  Let $d = \gcd(2,b - 1)$.
A finite set $T$ of positive integers is $(2,b)$-good if and only if
all of the elements of $T$ are congruent modulo $d$.
\end{theorem}

\begin{proof}
Fix a finite set of positive integers, $T$. 
First, assume that $T$ is $(2,b)$-good.  If $b$ is even, then $d = 1$, and the congruence result is trivial.  If $b$ is odd, fix $u\in U_{2,b}$. Then there exists $n$, $k \in {\Z}^+$ such that for each $t \in T$, $S_{2,b}^k(t+n) = u$.   It follows from Lemma~\ref{paritylemma}
that, for each $t \in T$, $t + n \equiv u \pmod 2$.  Hence, the elements of $T$ are congruent modulo $d = 2$.  

For the converse, assume that the elements of $T$
are congruent modulo $d$.
If $T$ is empty, then vacuously it is $(2,b)$-good. If $T = \{t\}$, then given $u \in U_{2,b}$, by definition, there exist $x\in \Z^+$ such that
$S_{2,b}(x) = u$.  
Fix some $r \in 2\Z^+$ such that $t \leq b^rx$.
Then, letting $n = b^rx-t$ and $k = 1$, since $S_{2,b}^k(t + n) = S_{2,b}(t+(b^rx-t)) = S_{2,b}(x) = u$, $T$ is $(2,b)$-good.

Now assume that $|T| = N > 1$ and assume, by induction, that any set of fewer than $N$ elements all
of which are congruent modulo $d$ is $(2,b)$-good. 
Let $t_1 > t_2 \in T$. By Lemmas~\ref{Flemmas} and~\ref{Ftheoremodd}, there exists a function $F$ as in Lemma~\ref{Flemma} such that
$F(t_1) = F(t_2)$.  This implies that $F(T)$ has fewer than $N$ elements. 
Further, since the elements of $T$ are congruent modulo $d$, the same
holds for $I(T)$ and, by Lemma~\ref{paritylemma},
for $S_{2,b}(T)$,
implying that the same holds for $F(T)$.
Thus, by the induction hypothesis, $F(T)$ is $(2,b)$-good and so, by Lemma~\ref{Flemma}, $T$ is $(2,b)$-good.
\end{proof}

\begin{cor}
\label{finalcor}
For $b \leq -3$ odd or $b \in \{-4,-6,-8,-10\}$ and $d = \gcd(2,b - 1)$, there exist arbitrarily long finite sequences of $d$-consecutive $b$-happy numbers.
\end{cor}

\begin{proof}
By Lemma~\ref{-2-3}, every odd positive integer is $-3$-happy.  So the corollary holds for $b = -3$.

For $b < -3$, given $N\in \Z^+$, let $T = \{1 + dt \;|\; 0\leq t \leq N-1\}$.
By Theorem~\ref{mainthm}, $T$ is $(2,b)$-good.  
By Definition~\ref{good}, there exist $n$, $k \in {\Z}^+$ such that for each $t \in T$, $S_{2,b}^k(t+n) = 1$.
Thus, $\{1 + n + dt\; |\; 0\leq t \leq N-1\}$ is a sequence of $N$ $d$-consecutive $b$-happy numbers, as desired.
\end{proof}

\end{document}